\documentclass{amsart}
\usepackage{latexsym}
\usepackage{amstext}
\usepackage{amsmath}
\usepackage{amssymb}
\usepackage{amsthm}
\usepackage{url}

\theoremstyle{plain}
\newtheorem{theorem}{Theorem}[section]

\newtheorem{proposition}[theorem]{Proposition}

\newtheorem{def-thm}[theorem]{Definition-Theorem}
\newtheorem{lemma}[theorem]{Lemma}

\theoremstyle{definition}

\newtheorem{example}[theorem]{Example}

\newtheorem{conjecture}[theorem]{Conjecture}

\def \kbar {\overline{k}}
\def \O {\mathcal{O}}

\def \AA {\mathcal{A}}
\def \QQ {\mathbb{Q}}
\def \RR {\mathbb{R}}

\DeclareMathOperator{\Supp}{Supp}

\DeclareMathOperator{\Pic}{Pic}
\DeclareMathOperator{\NS}{NS}

\DeclareMathOperator{\codim}{codim}

\DeclareMathOperator{\Span}{Span}

\allowdisplaybreaks
\begin{document}

\title[Degeneracy of integrals points]{On the degeneracy of integral points and entire curves in the complement of nef effective divisors}

\author{Gordon Heier}\author{Aaron Levin}
\address{Department of Mathematics\\University of Houston\\4800 Calhoun Road\\ Houston, TX 77204\\USA}
\email{heier@math.uh.edu}

\address{Department of Mathematics\\Michigan State University\\619 Red Cedar Road\\ East Lansing, MI 48824\\USA}
\email{adlevin@math.msu.edu}
\thanks{The second author was supported in part by NSF grants DMS-1352407 and DMS-2001205.}



\begin{abstract}
As a consequence of the divisorial case of our recently established generalization of Schmidt's subspace theorem, we prove a degeneracy theorem for integral points on the complement of a union of nef effective divisors.  A novel aspect of our result is the attainment of a strong degeneracy conclusion (arithmetic quasi-hyperbolicity) under weak positivity assumptions on the divisors. The proof hinges on applying our recent theorem with a well-situated ample divisor realizing a certain lexicographical minimax. We also explore the connections with earlier work by other authors and make a conjecture regarding bounds for the numbers of divisors necessary, including consideration of the question of arithmetic hyperbolicity. Under the standard correspondence between statements in Diophantine approximation and Nevanlinna theory, one obtains analogous degeneration statements for entire curves.
\end{abstract}

\maketitle

\section{Introduction}

Siegel's theorem on integral points on affine curves asserts that an affine curve $C$ over a number field $k$ has only finitely many integral points if $C$ has at least $3$ points at infinity (over $\kbar$).  This statement implies the more usual version of Siegel's theorem which requires the condition at infinity only if $C$ is rational (e.g., see \cite[Remark 7.3.10]{BG}).  A new line of results opened up when Corvaja and Zannier \cite{CZ02} gave a novel proof of Siegel's theorem using Schmidt's subspace theorem from Diophantine approximation.  Following subsequent work of Corvaja and Zannier \cite{CZ04b}, the second author proved the following generalization of Siegel's theorem to surfaces.

\begin{theorem}[{\cite[Theorem 11.5A]{levin_annals}}]
\label{Levin}
Let $X$ be a non-singular projective surface defined over a number field $k$.  Let $D_1,\ldots, D_q$ be effective ample divisors on $X$, defined over $k$, in general position and let $D=\sum_{i=1}^qD_i$.  
\begin{enumerate}
\item If $q\geq 4$ then $X\setminus D$ is arithmetically quasi-hyperbolic.
\item If $q\geq 5$ then $X\setminus D$ is arithmetically hyperbolic.
\end{enumerate}
\end{theorem}

The conclusion of arithmetic quasi-hyperbolicity means roughly that $S$-integral points on $X\setminus D$ are contained (up to finitely many points) in a proper closed subset $Z\subset X$ which is {\it geometric}, that is, independent of the number field and set of places $S$.  More formally, given a variety $V=X\setminus D$ defined over a number field $k$, we say that $V$ is {\it arithmetically quasi-hyperbolic} if there exists a proper closed subset $Z\subset X$ such that for every number field $k'\supset k$, every finite set of places $S$ of $k'$ containing the archimedean places, and every set $R$ of ($k'$-rational) $(D,S)$-integral points on $X$, the set $R\setminus Z$ is finite.  We say that $X\setminus D$ is {\it arithmetically hyperbolic} if all sets of $(D,S)$-integral points on $X$ are finite (i.e., one may take $Z=\emptyset$ in the definition of quasi-hyperbolicity). We refer the reader to \cite[Ch.~1,~\S4]{Vojta_LNM} for the notion of $(D,S)$-integral sets of points.  If $X$ is a projective variety of dimension $n$, we say that effective (possibly reducible) Cartier divisors $D_1,\ldots, D_q$ on $X$ are in {\it general position} if for any subset $I\subset\{1,\ldots, q\}$ with $|I|\leq n+1$ we have $\codim \cap_{i\in I} \Supp D_i\geq |I|$, where $\Supp D_i$ denotes the support of $D_i$ and we use the convention that $\dim \emptyset=-1$.

In general, a conjecture of the second author \cite[Conjecture 5.4A]{levin_annals} (slightly modified) states:

\begin{conjecture}
\label{LevConj}
Let $X$ be a projective variety, defined over a number field $k$, of dimension $n$.  Let $D_1,\ldots, D_q$ be effective ample Cartier divisors on $X$, defined over $k$, in general position, and let $D=\sum_{i=1}^qD_i$.  
\begin{enumerate}
\item If $q\geq n+2$, then $X\setminus D$ is arithmetically quasi-hyperbolic. \label{conjparta}
\item If $q\geq 2n+1$, then $X\setminus D$ is arithmetically hyperbolic.\label{conjpartb}
\end{enumerate}
\end{conjecture}

It was also observed in \cite{levin_annals} that when $X$ is non-singular and $D$ has normal crossings, part \eqref{conjparta} of the conjecture follows from (Bombieri-Lang-)Vojta's conjecture on the quasi-hyperbolicity of varieties of log general type and Mori theory \cite[Lemma 1.7]{Mori}.

Shortly after work of Corvaja, Levin, and Zannier \cite{CLZ}, Autissier proved the following result towards Conjecture \ref{LevConj}\eqref{conjparta}: 

\begin{theorem}[{\cite[Th\'eor\`eme 1.3, Remarque 2.3]{Aut11}}]
\label{Autissier}
Let $X$ be a Cohen-Macaulay projective variety, defined over a number field $k$, of dimension $n\geq 2$.  Let $D_1,\ldots, D_q$ be effective ample Cartier divisors on $X$, defined over $k$, in general position and let $D=\sum_{i=1}^qD_i$.  If 
\begin{align*}
q\geq 2n,
\end{align*}
then $X\setminus D$ is arithmetically quasi-hyperbolic.
\end{theorem}

Towards Conjecture \ref{LevConj}\eqref{conjpartb}, we have:

\begin{theorem}
\label{Levin2}
Under the hypotheses of Conjecture \ref{LevConj}, if $n\geq 2$ and
\begin{align*}
q\geq 2n^2,
\end{align*}
then $X\setminus D$ is arithmetically hyperbolic.
\end{theorem}

This was proved in \cite[Theorem 9.11A]{levin_annals} assuming the inequality $q\geq 2n^2+1$.  The slight improvement given here comes from applying the same proof as in \cite{levin_annals}, but with an improved estimate of Autissier \cite[Lemme 4.2, Corollaire 4.3]{Aut09}. 

It is essential in Theorem \ref{Autissier} that the divisors satisfy ampleness or some other positivity condition of similar strength.  Indeed, if $X$ contains a Zariski dense set of $D$-integral points, then by blowing up points in $D$, one obtains a variety $\tilde{X}$ and a divisor $\tilde{D}$ on $\tilde{X}$ with an arbitrarily large number of components and $X\setminus D\cong \tilde{X}\setminus \tilde{D}$ (and hence there will be a Zariski dense set of $\tilde{D}$-integral points on $\tilde{X}$).  Thus, without a positivity assumption of some sort, there is no inequality on the number of components $q$ sufficient to guarantee Zariski non-density of integral points.  However, as is well known, each time we blow up the variety $X$ the rank of the Picard group increases by one.  Taking into account the rank of the subgroup in $\Pic X$ generated by $D_1,\ldots, D_q$, Vojta proved: 

\begin{theorem}[{\cite[Theorem 2.4.1]{Vojta_LNM}}]
\label{Vojta1}
Let $X$ be a projective variety, defined over a number field $k$, of dimension $n$.  Let $D=\sum_{i=1}^qD_i$ be a sum of distinct prime Cartier divisors on $X$ defined over $k$.  Let $r$ be the rank of the subgroup in $\Pic X$ generated by $D_1,\ldots, D_q$.  If 
\begin{align*}
q\geq  n+r+1,
\end{align*}
then all sets of $(D,S)$-integral points on $X$ are not Zariski dense.
\end{theorem}

More generally, as an application of results on integral points on semiabelian varieties, Vojta proved a result depending on the rank in the N\'eron-Severi group $\NS X$. 
\begin{theorem}[{\cite[Corollary 0.3]{vojta_inv_math_1996}}]
\label{Vojta2}
Let $X$ be a projective variety, defined over a number field $k$, of dimension $n$.  Let $D=\sum_{i=1}^qD_i$ be a sum of distinct prime Cartier divisors on $X$ defined over $k$.  Let $r$ be the rank of the subgroup in $\NS X$ generated by $D_1,\ldots, D_q$.  If 
\begin{align*}
q\geq  n+r-h^1(X,\O_X)+1,
\end{align*}
then all sets of $(D,S)$-integral points on $X$ are not Zariski dense.
\end{theorem}

In both Theorems \ref{Vojta1} and \ref{Vojta2} it is easy to see (e.g., from Examples \ref{P2lines} and \ref{P1squared}) that the conclusions cannot be strengthened to quasi-hyperbolicity statements.  

Under a combined ampleness and general position assumption, Noguchi and Winkelmann proved a finiteness statement.

\begin{theorem}[{\cite[Theorem 9.7.6]{NW}}]
\label{NW}
Let $X$ be a projective variety, defined over a number field $k$, of dimension $n$.  Let $D=\sum_{i=1}^qD_i$ be a sum of ample effective Cartier divisors in general position on $X$ defined over $k$.  Let $r$ be the rank of the subgroup in $\NS X$ generated by $D_1,\ldots, D_q$.  If 
\begin{align*}
q\geq 2n+r,
\end{align*}
then $X\setminus D$ is arithmetically hyperbolic.
\end{theorem}

It should be pointed out that we have stated the above three theorems in terms of ranks associated to the given divisors $D_1,\ldots, D_q$, while these results are mostly stated  in the literature in terms of absolute invariants (e.g., the Picard number) which are independent of the given divisors.

In this note, we initiate the study of arithmetic (quasi-)hyperbolicity in the context of nef divisors.  From one point of view, our main result is in the vein of Theorems \ref{Vojta1}--\ref{NW}, with the rank replaced by an appropriate analogous quantity involving the number of generators of the cone in the real N\'eron-Severi vector space generated by the divisors $D_i$.  From another point of view, as discussed below, the main result goes towards a version of Conjecture \ref{LevConj} for nef divisors.  We now state the main result, yielding (quasi-)hyperbolicity statements under weak positivity assumptions on the divisors.  We use $\equiv$ to denote numerical equivalence of integral as well as $\QQ$- and $\RR$-divisors (see \cite[Ch.\ 1.3]{PAGI}).

\begin{theorem}
\label{mthm}
Let $X$ be a projective variety, defined over a number field $k$, of dimension $n$.  Let $E_1,\ldots, E_r$ be nef Cartier divisors on $X$ with $\sum_{j=1}^rE_j$ ample.  Let $D_1,\ldots, D_q$ be non-zero effective (possibly reducible) Cartier divisors in general position on $X$ and let $D=\sum_{i=1}^qD_i$. Suppose that $D_i\equiv \sum_{j=1}^ra_{i,j}E_j$, $i=1,\ldots, q$, where the coefficients $a_{i,j}$ are non-negative real numbers.  Let $P_i=(a_{i,1},\ldots, a_{i,r})\in \mathbb{R}^r$, $i=1,\ldots, q$.    Assume that for any proper subset $T$ of the set of standard basis vectors $\{e_1,\ldots, e_r\}\subset\mathbb{R}^r$, at most $(\#T)\left\lfloor\frac{q}{r}\right\rfloor$ of the vectors $P_1,\ldots, P_q$ are supported on $T$.  
\begin{enumerate}
\item If
\begin{align*}
q&\geq r(n+1)+1, && r=1,2,\\
q&\geq r(n+1)+\frac{(r-1)(r-2)}{2}, && r\geq 3,
\end{align*}
then $X\setminus D$ is arithmetically quasi-hyperbolic.
\label{mthma}
\item If
\begin{align*}
q\geq 2nr+r^2,
\end{align*}
then $X\setminus D$ is arithmetically hyperbolic.
\label{mthmb}
\end{enumerate}
\end{theorem}

Let $\mathcal{C}$ be the convex cone generated by the numerical equivalence classes of $E_1,\ldots, E_r$ in the real N\'eron-Severi vector space.  Then the classes of the divisors $D_i$ lie in $\mathcal{C}$, and the condition that $\sum_{j=1}^rE_j$ is ample is equivalent to the convex cone $\mathcal{C}$ containing an ample class. The condition involving the supports of the vectors $P_i$ in terms the standard basis of $\RR^r$ ensures that the classes of the divisors $D_i$ are sufficiently ``spread out" in the cone $\mathcal{C}$.  Some such condition is necessary to avoid counterexamples such as Example~\ref{P1squared} in Section \ref{examples}, where all of the numerical equivalence classes of the divisors are multiples of some non-ample class.

In view of Theorem \ref{mthm} and the results of Section \ref{proofb}, it seems reasonable to conjecture the following analogue of Conjecture \ref{LevConj}:

\begin{conjecture}
\label{conj2}
Assume the hypotheses of Theorem \ref{mthm}.
\begin{enumerate}
\item If
\begin{align*}
q\geq  r(n+1)+1,
\end{align*}
then $X\setminus D$ is arithmetically quasi-hyperbolic.
\label{conja}
\item If
\begin{align*}
q\geq  2nr+1,
\end{align*}
then $X\setminus D$ is arithmetically hyperbolic.
\label{conjb}
\end{enumerate}
\end{conjecture}

We show in Example \ref{conjbex} in Section \ref{examples} that the inequality in part \eqref{conjb} of the conjecture is best possible. We are not sure if the inequality in part (a) of the conjecture is best possible; however, in Example \ref{conjaex} we show that in general $r(n+1)$ cannot be replaced by anything better than $r(n-r+2)$.\par
Observe that Theorem \ref{mthm}\eqref{mthma} proves Conjecture \ref{conj2}\eqref{conja} when $r\leq 3$. Note that in Theorem \ref{mthm}\eqref{mthma}, despite the identity $(3-1)(3-2)/2 =1$, we have grouped the case $r=3$ together with the general case as the general method of proof starts to apply from $r=3$ onwards, with the cases $r=1,2$ being easy specializations of the general argument.  In general, we may view Theorem \ref{mthm} as approximating Conjecture \ref{conj2}, with the inequalities involving an ``error term" depending only on $r$. For arbitrary $r$, we suspect that Lemma \ref{mainl} in the next section holds true with a stronger conclusion (namely, $n_j(Q,P_1,\ldots, P_q)\geq \lfloor \frac{q}{r}\rfloor$ for $j=1,\ldots, r$) which would yield Conjecture \ref{conj2}\eqref{conja}. However, proving such improved inequalities seems to be a surprisingly difficult combinatorial problem. \par
When $r$ is large compared to the dimension $n$, we are able to obtain the following better bound.

\begin{theorem}
\label{mthm2}
Assume the hypotheses of Theorem \ref{mthm} and that $n\geq 2$.
\begin{enumerate}
\item If $X$ is Cohen-Macaulay and $q\geq 2nr$, then $X\setminus D$ is arithmetically quasi-hyperbolic.
\item If $q\geq 2n^2r$, then $X\setminus D$ is arithmetically hyperbolic.
\end{enumerate}
\end{theorem}
In Section 3, we will derive Theorem \ref{mthm}\eqref{mthmb} and Theorem \ref{mthm2} from Theorem \ref{Autissier},  Theorem \ref{Levin2}, and Theorem \ref{NW}.  The majority of the paper is devoted to the proof of Theorem \ref{mthm}\eqref{mthma}, which may be regarded as the primary new result.  Theorem \ref{mthm}\eqref{mthma} does not seem to na\"ively follow from previous results (and the method of Section \ref{proofb}), and in fact in certain cases gives a non-trivial improvement to Autissier's Theorem~\ref{Autissier}.  For instance, when $r\leq 4$, Theorem \ref{mthm}\eqref{mthma} implies Conjecture \ref{LevConj}\eqref{conjparta} when each ample divisor $D_i$ splits as a sum of $r$ non-zero effective nef divisors which satisfy, in totality, the hypotheses of Theorem \ref{mthm} (and when $r\geq 5$, Theorem \ref{mthm}\eqref{mthma} implies, under similar hypotheses,  arithmetic quasi-hyperbolicity on the complement of $q\geq  n+1+(r-2)/2$ ample effective divisors).  We discuss a further application of Theorem \ref{mthm}\eqref{mthma} in Example \ref{appex}.

The proof of Theorem \ref{mthm}\eqref{mthma} is based on the following result from our recent work \cite{HL17}.

\begin{theorem}
\label{Cor2}
Let $X$ be a projective variety of dimension $n$ defined over a number field $k$.  Let $S$ be a finite set of places of $k$.  Let $D_1,\ldots, D_q$ be effective Cartier divisors on $X$, defined over $k$, and in general position.  Let $A$ be an ample Cartier divisor on $X$, and $\epsilon>0$. Let $c_i$ be rational numbers such that $A-c_iD_i$ is a nef $\QQ$-divisor for all $i$. Then there exists a proper Zariski closed subset $Z\subset X$, independent of $k$ and $S$, such that for all but finitely many points $P\in X(k)\setminus Z$,
\begin{equation*}
\sum_{i=1}^q c_i m_{D_{i},S}(P)< (n+1+\epsilon)h_A(P).
\end{equation*}
\end{theorem}

Here, $m_{D,S}(P)=\sum_{v\in S} \lambda_{D,v}(P)$ is a sum of local height functions $\lambda_{D,v}$, associated to the divisor $D$ and place $v$ in $S$, and $h_A$ is a global (absolute) height associated to $A$.

Theorem \ref{Cor2} may be viewed as a generalization of work of Evertse and Ferretti \cite{ef_festschrift} and Corvaja and Zannier \cite{CZ}, which dealt with the case when the divisors $A,D_1,\ldots, D_q$ have a common multiple up to linear equivalence (or work of the second author \cite{levin_duke} when the divisors have a common multiple up to numerical equivalence).  More generally, building on the work of Evertse and Ferretti \cite{ef_festschrift}, Corvaja and Zannier \cite{CZ}, McKinnon and Roth \cite{McK_R} and others, a version of Theorem \ref{Cor2} was proved in \cite{HL17} for closed subchemes (in place of divisors) and with the constants $c_i$ replaced by suitably-defined Seshadri constants.  The fact that $Z$ can be chosen independently of $k$ and $S$ in Theorem \ref{Cor2} (and its generalizations) relies on Vojta's result \cite{vojta_ajm_87} on the exceptional set in Schmidt's subspace theorem, and that the proof of Theorem \ref{Cor2} ultimately relies on an application of Schmidt's theorem.

The proof of Theorem \ref{mthm}\eqref{mthma} proceeds through Theorem \ref{Cor2}, and takes advantage of the freedom in choosing the ample divisor $A$ in Theorem \ref{Cor2}.  Roughly speaking, the idea of the proof of Theorem \ref{mthm}\eqref{mthma} is to choose an ample divisor $A$ in Theorem \ref{Cor2} whose image in the relevant convex cone $\mathcal{C}$ is centrally located relative to the classes of $D_1,\ldots, D_q$ in $\mathcal{C}$.  In practice, we achieve this by choosing an $A$ which achieves a certain lexicographical minimax.

Under the standard correspondence between statements in Diophantine approximation and Nevanlinna theory, there exist analogous degeneration statements for entire curves in Nevanlinna theory. This line of reasoning is by now well known and we omit the details. 

\section{Proof of Theorem \ref{mthm}\eqref{mthma}}

The proof of Theorem \ref{mthm}\eqref{mthma} is based on the following proposition.

\begin{proposition}\label{keyprop}
Let $X$ be a projective variety of dimension $n$ defined over a number field $k$.  Let $E_1,\ldots, E_r$ be nef Cartier divisors on $X$ with $\sum_{j=1}^rE_j$ ample. Let $D_1,\ldots, D_q$ be non-zero effective (possibly reducible) Cartier divisors in general position on $X$. Suppose that $D_i\equiv \sum_{j=1}^ra_{i,j}E_j$, $i=1,\ldots, q$, where the coefficients $a_{i,j}$ are non-negative real numbers.  Let $P_i=(a_{i,1},\ldots, a_{i,r})\in \mathbb{R}^r$, $i=1,\ldots, q$.    Assume that for any proper subset $T$ of the set of standard basis vectors $\{e_1,\ldots, e_r\}\subset\mathbb{R}^r$, at most $(\#T)\left\lfloor\frac{q}{r}\right\rfloor$ of the vectors $P_1,\ldots, P_q$ are supported on $T$. If 
\begin{align*}
q&\geq r(n+1)+1, && r=1,2,\\
q&\geq r(n+1)+\frac{(r-1)(r-2)}{2}, && r\geq 3,
\end{align*}
then there exist an ample divisor $A$ and positive rational constants $c_1,\ldots,c_q, \delta$ such that for all $i=1,\ldots, q$: 
\begin{align*}
A-c_iD_i&\text{ is $\QQ$-nef}
\end{align*}
and
\begin{align*}
\sum_{i=1}^q c_iD_i-(n+1+\delta)A &\text{ is  $\QQ$-nef.}
\end{align*}
\end{proposition}
Assuming Proposition \ref{keyprop}, the proof of Theorem \ref{mthm}\eqref{mthma} proceeds as follows.
\begin{proof}[Proof of Theorem \ref{mthm}(a)]
Let $A$, $c_1,\ldots, c_q$, and $\delta$ be as in the conclusion of Proposition~\ref{keyprop}. Let $\epsilon<\delta$ be a positive rational number.  First, note that
\begin{align*}
\sum_{i=1}^q c_iD_i-(n+1+\epsilon)A =  \sum_{i=1}^q c_iD_i-(n+1+\delta)A +(\delta-\varepsilon)A
\end{align*}
is an ample $\QQ$-divisor, as it is the sum of a nef $\QQ$-divisor (by Proposition \ref{keyprop}) and an ample $\QQ$-divisor.  Now, since $A-c_iD_i$ is $\QQ$-nef for all $i=1,\ldots,q$, by Proposition \ref{keyprop}, we may apply Theorem \ref{Cor2} to conclude that there exists a proper Zariski closed subset $Z\subset X$, independent of $k$ and $S$, such that for all $P\in X(k)\setminus Z$,
\begin{equation*}
\sum_{i=1}^q c_im_{D_i,S}(P)< (n+1+\epsilon)h_A(P).
\end{equation*}
Furthermore, if $R\subset X(k)$ is a set of $(D,S)$-integral points on $X$, then for $P\in R\setminus Z$,
\begin{equation*}
\sum_{i=1}^q c_im_{D_i,S}(P)=\sum_{i=1}^q c_ih_{D_i}(P)+O(1)< (n+1+\epsilon)h_A(P)+O(1).
\end{equation*}
Since $\sum_{i=1}^q c_iD_i-(n+1+\epsilon)A$ is $\QQ$-ample, by Northcott's theorem the inequality $\sum_{i=1}^q c_ih_{D_i}(P)< (n+1+\epsilon)h_A(P)+O(1)$ has only finitely many solutions $P\in X(k)$.  It follows that $R\setminus Z$ is finite.
\end{proof}
It remains to prove Proposition \ref{keyprop}. To this end, we establish the following lemma. Note that we naturally interpret division of a positive number by zero as (positive) infinity. 
\begin{lemma}
\label{mainl}
Let $P_i=(a_{i,1},\ldots, a_{i,r})\in \mathbb{R}^r\setminus \{0\}$, $i=1,\ldots, q$, be vectors with non-negative coordinates.  Let $e_j$, $j=1,\ldots, r,$ be the standard coordinate vectors.  Suppose that for any proper subset $T\subset \{e_1,\ldots, e_r\}$ of cardinality $t$, at most $t\left\lfloor\frac{q}{r}\right\rfloor$ of the vectors $P_i$ are supported on $T$.  For $Q=(b_1,\ldots, b_r)\in \mathbb{R}^r$ with positive coordinates, define
\begin{align*}
n_j(Q,P_1,\ldots, P_q)=\#\left\{i\in \{1,\ldots,q\}\mid \min_{l=1,\ldots,r} \frac{b_l}{a_{i,l}}=\frac{b_j}{a_{i,j}}\right\}, \quad j=1,\ldots, r.
\end{align*}
Assume additionally that for all $i\neq i',j\neq j'$, we have 
\begin{equation}\label{generic}
a_{i,j}a_{i',j'}-a_{i,j'}a_{i',j}\neq 0,
\end{equation}
unless both terms on the left are $0$.  Then there exists $Q=(b_1,\ldots, b_r)\in \mathbb{Q}^r$ with positive coordinates such that 
\begin{align*}
n_j(Q,P_1,\ldots, P_q)\geq \frac{q}{r}-\frac{r-1}{2}, \quad j=1,\ldots, r,
\end{align*}
and
\begin{align}
\left(\frac{1}{2}\min \frac{a_{i,j}}{a_{i',j'}}\right)^r\leq \frac{b_l}{b_{l'}}\leq \left(2\max \frac{a_{i,j}}{a_{i',j'}}\right)^r, \quad \text{ for all }l, l', \label{bineq}
\end{align}
where the minimum and maximum are taken over all $i,j,i',j'$ such that $a_{i,j}a_{i',j'}\neq 0$.
\end{lemma}

\begin{proof}
To a point $Q=(b_1,\ldots, b_r)\in \mathbb{R}^r$ with positive coordinates, we associate the point $n(Q)=(n_1,\ldots, n_r)\in \mathbb{N}^r$, where $n_j=n_j(Q)=n_j(Q,P_1,\ldots, P_q)$, $j=1,\ldots, r$.  Let $\AA\subset\mathbb{R}^r$ be the subset of $Q=(b_1,\ldots, b_r)$ with positive coordinates such that 
\begin{enumerate}
\item The non-zero coordinates of the vector $(a_{i,1}/b_1,\ldots, a_{i,r}/b_r)$ are distinct for any fixed $i=1,\ldots, q$.\label{cond1}
\item The ratios of all distinct non-zero coordinates of $(a_{i,1}/b_1,\ldots, a_{i,r}/b_r)$ (over all $i$) are distinct.\label{cond2}
\end{enumerate}

Then $\AA$ is clearly an open subset of $\mathbb{R}^r$. By condition \eqref{generic}, $\AA$ is non-empty. The condition \eqref{cond1} ensures that for $Q\in \AA$, every point $P_i$ contributes to a unique $n_j(Q,P_1,\ldots, P_q)$.  In particular, for $Q\in \AA$,  $\sum_{j=1}^rn_j(Q,P_1,\ldots, P_q)=q$.\par  
We consider $\mathbb{N}^r$ with the usual lexicographical ordering.  Let $Q\in \AA$ be such that it realizes the lexicographical minimax 
$$\min_{P\in \AA}\ \max\{\sigma(n(P)): \sigma \in S_r\},$$
where $S_r$ is the symmetric group on $r$ letters.  After permuting the coordinates, we can assume without loss of generality that $n(Q)=(n_1,\ldots, n_r)$ satisfies $n_1\geq n_2\geq \ldots \geq n_r$. \par

We claim that $n_j-n_{j+1}\in \{0,1\}$ for $1\leq j\leq r-1$.  Suppose otherwise, and let $j_0$ be the smallest index such that $n_{j_0}-n_{j_0+1}\geq 2$.  We consider the family of points
\begin{align*}
Q_\lambda=(\lambda b_1,\ldots, \lambda b_{j_0}, b_{j_0+1},\ldots, b_r)=(b_{\lambda, 1},\ldots, b_{\lambda,r}), \quad \lambda \geq 1.
\end{align*}

By assumption, there are at most $j_0\left\lfloor\frac{q}{r}\right\rfloor$ vectors $P_i$ supported on $e_1,\ldots, e_{j_0}$.  Since $\sum_{j=1}^{j_0}n_j>j_0\left\lfloor\frac{q}{r}\right\rfloor$, this implies that for some $\lambda>1$, $n(Q_\lambda)\neq n(Q)$.  Condition \eqref{cond1} implies that there is a minimal such value $\lambda>1$.  From the form of $Q_\lambda$ and condition \eqref{cond2}, for this value of $\lambda$ there is a unique $j_1\leq j_0$, $j_2>j_0$, and $i$ such that
\begin{align*}
\min_{l=1,\ldots,r}  \frac{b_{\lambda,l}}{a_{i,l}}=\frac{\lambda b_{j_1}}{a_{i,j_1}}=\frac{ b_{j_2}}{a_{i,j_2}}.
\end{align*}
Then for sufficiently small $\epsilon>0$, $Q_{\lambda+\epsilon}\in \AA$, $n_{j_1}(Q_{\lambda+\epsilon})=n_{j_1}(Q)-1$, $n_{j_2}(Q_{\lambda+\epsilon})=n_{j_2}(Q)+1$, and $n_{j}(Q_{\lambda+\epsilon})=n_{j}(Q)$ if $j\not\in \{j_1,j_2\}$.  Since $n_{j_1}-n_{j_2}\geq n_{j_0}-n_{j_0+1}\geq 2$, this implies that
\begin{align*}
\max\{\sigma(n(Q_{\lambda+\epsilon})): \sigma \in S_r\}< n(Q),
\end{align*}
contradicting the definition of $Q$ and proving the claim.

Now, we note that $n_j-n_{j+1}\in \{0,1\}$ for $1\leq j\leq r-1$ implies the inequalities
\begin{align*}
rn_r\leq q=\sum_{j=1}^rn_j\leq rn_r+\frac{r(r-1)}{2}.
\end{align*}
The last inequality implies 
$$(n_1\geq n_2\geq \ldots \geq)\ n_r\geq  \frac q r - \frac{r-1}{2}.$$
Due to the condition (a) imposed on the set $\AA$, it is clear that we may replace $Q$ by a sufficiently close point with rational coefficients and maintain the above chain of inequalities. Lemma \ref{mainl} is now proven except for the bounds \eqref{bineq}.  By symmetry, it suffices to prove that we can choose $Q=(b_1,\ldots, b_r)$ satisfying
\begin{align*}
\frac{b_l}{b_{l'}}\leq \left(2\max \frac{a_{i,j}}{a_{i',j'}}\right)^r \quad \text{ for all }l, l',
\end{align*}
where the maximum is over all $i,j,i',j'$ such that $a_{i,j}a_{i',j'}\neq 0$.  Let $Q=(b_1,\ldots, b_r)$ be one choice of $Q$ satisfying the lemma except for possibly the inequality \eqref{bineq}.  For simplicity, after reindexing, we may assume that $0<b_1\leq b_2\leq \ldots \leq b_r$.  Suppose that for some index $l$,
\begin{align*}
\frac{b_{l+1}}{b_l}>2\max \frac{a_{i,j}}{a_{i',j'}}.
\end{align*}
Let $\lambda$ be a rational number satisfying
\begin{align*}
\frac{b_l}{b_{l+1}}\max \frac{a_{i,j}}{a_{i',j'}}< \lambda < 2\frac{b_l}{b_{l+1}}\max \frac{a_{i,j}}{a_{i',j'}}<1,
\end{align*}
and let
\begin{align*}
Q'=(b_1',\ldots, b_r')=(b_1,\ldots, b_l, \lambda b_{l+1},\lambda b_{l+2},\ldots, \lambda b_r).
\end{align*}
Note that $Q'$ again has positive rational coordinates. We claim that
\begin{align*}
n_j(Q,P_1,\ldots, P_q)\leq n_j(Q',P_1,\ldots, P_q), \quad j=1,\ldots, r.
\end{align*}

Let $j\in \{1,\ldots, r\}$ and $i\in \{1,\ldots,q\}$ be such that
\begin{align*}
\min_{m=1,\ldots,r} \frac{b_m}{a_{i,m}}=\frac{b_j}{a_{i,j}}.
\end{align*}
In particular, $a_{i,j}\neq 0$.  Suppose first that $j\leq l$.  For $m\geq l+1$ we have
\begin{align*}
\frac{b_m'}{b_j'}=\frac{\lambda b_m}{b_j} > \max \frac{a_{i',j'}}{a_{i,j}}.
\end{align*}
For $m\leq l$ we have
\begin{align*}
\frac{b_m'}{b_j'}=\frac{ b_m}{b_j}.
\end{align*}
It follows that
\begin{align*}
\min_{m=1,\ldots,r} \frac{b_m'}{a_{i,m}}=\frac{b_j'}{a_{i,j}}=\frac{b_j}{a_{i,j}}.
\end{align*}

Suppose now that $j\geq l+1$.  Let $m\leq l$.  If $a_{i,m}\neq 0$, then
\begin{align*}
\frac{b_m}{b_j}< \frac{1}{2}\min \frac{a_{i',j'}}{a_{i,j}}<\frac{a_{i,m}}{a_{i,j}},
\end{align*}
contradicting the choice of $i$ and $j$.  Therefore $a_{i,m}=0$.  If $m\geq l+1$ then
\begin{align*}
\frac{b_m'}{b_j'}=\frac{b_m}{b_j}.
\end{align*}
It follows that
\begin{align*}
\min_{m=1,\ldots,r} \frac{b_m'}{a_{i,m}}=\frac{b_j'}{a_{i,j}}=\lambda\frac{b_j}{a_{i,j}}.
\end{align*}
Therefore
\begin{align*}
n_j(Q',P_1,\ldots, P_q)\geq n_j(Q,P_1,\ldots, P_q)\geq \frac{q}{r}-\frac{r-1}{2}, \quad j=1,\ldots, r,
\end{align*}
and replacing $Q$ by $Q'$, we now have the inequality
\begin{align*}
\frac{b_{l+1}}{b_l}\leq 2\max \frac{a_{i,j}}{a_{i',j'}}.
\end{align*}

Repeating this argument finitely many times, we find a suitable $Q=(b_1,\ldots, b_r)$ with positive rational coordinates such that for $l=1,\ldots, r-1$,  
\begin{align*}
\frac{b_{l+1}}{b_l}\leq 2\max \frac{a_{i,j}}{a_{i',j'}},
\end{align*}
which implies \eqref{bineq}.\end{proof}
\begin{proof}[Proof of Proposition \ref{keyprop}]

We take $$\alpha_{1,1}(\kappa),\ldots,\alpha_{1,r}(\kappa),\alpha_{2,1}(\kappa),\ldots,\alpha_{2,r}(\kappa),\ldots,\alpha_{q,1}(\kappa),\ldots,\alpha_{q,r}(\kappa)$$
to be (discontinuous) functions of $\kappa\in (0,1]$ with the following properties. The function $\alpha_{i,j}(\kappa)$ is identically equal to $0$ if $a_{i,j}=0$. If, on the other hand, $a_{i,j}\not =0$, then $\alpha_{i,j}(\kappa)$ takes on positive real values such that we have the limits
$$\lim_{\kappa\searrow 0} \alpha_{i,j}(\kappa) = 0.$$
Moreover, the $\RR$-divisors $B_i(\kappa) = \alpha_{i,1}(\kappa)E_1+\ldots+ \alpha_{i,r}(\kappa)E_r$ are such that
$$D_i'(\kappa):=D_i+B_i(\kappa)\equiv \sum_{j=1}^r a'_{i,j}(\kappa) E_j,\quad i=1,\ldots,q,$$
have rational coefficients $a'_{i,j}(\kappa) = a_{i,j} + \alpha_{i,j}(\kappa)$ and the vectors
$$P'_1(\kappa)=(a'_{1,1}(\kappa),\ldots, a'_{1,r}(\kappa)),\ldots, P'_q(\kappa)=(a'_{q,1}(\kappa),\ldots, a'_{q,r}(\kappa))$$ 
satisfy the assumptions of Lemma \ref{mainl}. Therefore, we can conclude that, for all $\kappa$, there exists a vector $Q'(\kappa)=(b'_1(\kappa),\ldots,b'_r(\kappa))$ as in Lemma~\ref{mainl} with respect to $P'_1(\kappa),\ldots,P'_q(\kappa)$.  We normalize the coordinates so that $b'_1=1$.  Then from the definitions and Lemma \ref{mainl}, for a sufficiently small choice of $\hat \kappa>0$ (we now fix one such choice), there exist positive rational constants $\gamma_1, \gamma_2, \gamma_3$, and $\gamma_4$ such that for all $0<\kappa<\hat \kappa$,
\begin{align*}
\gamma_1<a_{i,j}'(\kappa)<\gamma_2
\end{align*}
for all $i$ and $j$ such that $a_{i,j}'(\kappa)\neq 0$ (or equivalently, $a_{i,j}\neq 0$), and
\begin{align*}
\gamma_3<b_j'(\kappa)<\gamma_4, \quad j=1,\ldots, r.
\end{align*}
We now choose a fixed positive rational number $\delta<\frac{\gamma_1\gamma_3}{2\gamma_2\gamma_4}$ and a fixed $0< \kappa_0=\kappa(\delta)<\hat\kappa $ such that
\begin{align}\label{delta_div}
\delta\gamma_3\sum_{j=1}^r E_j-\frac{\gamma_4}{\gamma_1} \sum_{i=1}^qB_i(\kappa_0)
\end{align}
is $\QQ$-nef.  We now set $Q'=Q'(\kappa_0)=(b'_1,\ldots,b'_r)$ with $b_1'=1$ and let
\begin{align*}
A=b'_1E_1+\ldots+ b'_rE_r.
\end{align*}
Then $A$ is $\QQ$-ample.  \par
We define positive rational numbers
\begin{align*}
c_i:=\min_{j=1,\ldots,r} \frac{b_j'}{a_{i,j}'(\kappa_0)}< \frac{\gamma_4}{\gamma_1},\quad i=1,\ldots,q.
\end{align*}
For $\RR$-divisors $F_1$ and $F_2$, we write $F_1\geq F_2$ if the difference $F_1-F_2$ is a nef $\RR$-divisor.  Then 
\begin{align*}
A-c_iD_i&\geq A-c_iD_i'(\kappa_0)\\
&\equiv \sum_{j=1}^r (b_j'- c_i a_{i,j}'(\kappa_0))E_j,
\end{align*}
which implies that $A-c_iD_i$ is a nef $\QQ$-divisor for $i=1,\ldots,q$.\par
We now deal only with the general case $r\geq 3$, as the cases $r=1,2$ are easy specializations of the following argument.\par
Since 
\begin{align*}
q\geq r(n+1)+\frac{(r-1)(r-2)}{2}=rn+\frac{r(r-1)}{2}+1, 
\end{align*}
we have
\begin{align*}
n_j(Q',P_1'(\kappa_0),\ldots, P_q'(\kappa_0))\geq \frac q r - \frac{r-1} 2> n, \quad j=1,\ldots, r.
\end{align*}
Therefore,
\begin{align}
\label{njeq}
n_j(Q',P_1'(\kappa_0),\ldots, P_q'(\kappa_0))\geq n+1, \quad j=1,\ldots, r,
\end{align}
as $n_j(Q',P_1'(\kappa_0),\ldots, P_q'(\kappa_0))$ is an integer.  Let $j\in \{1,\ldots, r\}$.  By hypothesis, at most $(r-1)\lfloor \frac{q}{r}\rfloor\leq q-\frac{q}{r}$ of the vectors $P_1,\ldots, P_q$ lie in $\Span(\{e_1,\ldots, e_r\}\setminus \{e_j\})$.  Since $q>r(n+1)$, it follows that there are at least $\lceil\frac{q}{r}\rceil\geq n+2$ points $P_i'(\kappa_0)$ with $a_{i,j}'(\kappa_0)>0$.  Combined with \eqref{njeq}, this implies that
\begin{align*}
\sum_{i=1}^q c_iD_i'(\kappa_0)&\geq \sum_{j=1}^r (n+1)b'_jE_j+\sum_{j=1}^r \left(\min_i c_i\right)\left(\min_{\substack{i, a_{i,j}'(\kappa_0)\neq 0}}a_{i,j}'(\kappa_0)\right)E_j\\
&\geq (n+1)A+\frac{\gamma_1\gamma_3}{\gamma_2}\sum_{j=1}^rE_j\\
&\geq \left(n+1+\frac{\gamma_1\gamma_3}{\gamma_2\gamma_4}\right)A.
\end{align*}

Finally, we find the inequalities
\begin{align*}
&\sum_{i=1}^q c_iD_i-\left(n+1+\delta\right)A\\=& \sum_{i=1}^q c_iD_i'(\kappa_0)-(n+1+2\delta)A+\delta A-\sum_{i=1}^q c_iB_i (\kappa_0)\\
\geq & \left(n+1+\frac{\gamma_1\gamma_3}{\gamma_2\gamma_4}\right)A-(n+1+2\delta)A +\delta\gamma_3\sum_{j=1}^rE_j-\frac{\gamma_4}{\gamma_1}\sum_{i=1}^q B_i(\kappa_0)\\
\geq & \underbrace{ \left(\frac{\gamma_1\gamma_3}{\gamma_2\gamma_4}-2\delta\right)}_{>0} A,
\end{align*}
where the last inequality is due to \eqref{delta_div}. Therefore, 
\begin{align*}
\sum_{i=1}^q c_iD_i-\left(n+1+\delta\right)A
\end{align*}
is $\mathbb{Q}$-ample and in particular $\mathbb{Q}$-nef.  Finally, by rescaling the coefficients $b_j'$ appearing in $A$ (and rescaling the $c_i$ by the same factor), we can assume that $A$ is an ample divisor (and not just an ample $\mathbb{Q}$-divisor).
\end{proof}

\section{Proof of Theorem \ref{mthm}\eqref{mthmb} and Theorem \ref{mthm2}}
\label{proofb}

We use the following simple lemma.

\begin{lemma}
\label{thmblem}
Let $P_i\in \mathbb{R}^r\setminus \{0\}$, $i=1,\ldots, q$, be vectors with non-negative coordinates.  Let $e_1,\ldots, e_r$ be the standard coordinate vectors.  Suppose that for any proper subset $T\subset \{e_1,\ldots, e_r\}$ of cardinality $t$, at most $t\left\lfloor\frac{q}{r}\right\rfloor$ of the vectors $P_i$ are supported on $T$.  Then there exist pairwise disjoint subsets $I_1,\ldots, I_{\left\lfloor \frac{q}{r}\right\rfloor}\subset \{1,\ldots, q\}$ of cardinality $r$ such that the vector
\begin{align*}
\sum_{i\in I_j}P_i
\end{align*}
has positive coordinates for $j=1,\ldots, \left\lfloor \frac{q}{r}\right\rfloor$.
\end{lemma}

\begin{proof}
We prove the result by induction on the dimension $r$.  For $r=1$ the result is trivial.  Suppose now that $r\geq 2$ and the result holds in dimension $r-1$.  By dropping some of the $P_i$ and replacing $q$ by $r \left\lfloor \frac{q}{r}\right\rfloor$, it suffices to prove the case that $q$ is divisible by $r$.  Let $\pi:\mathbb{R}^r\to\mathbb{R}^{r-1}$ denote the projection onto the first $r-1$ coordinates.    By hypothesis, there are at most $\frac{q}{r}$ vectors $P_i$ with $\pi(P_i)=0$, and hence at least
\begin{align*}
q':=q-\frac{q}{r}=\frac{q(r-1)}{r}
\end{align*}
vectors $P_i$ such that $\pi(P_i)\neq 0$.  Similarly, taking $t=r-1$, there are at most $q'$ vectors $P_i$ whose last coordinate is $0$ (and necessarily $\pi(P_i)\neq 0$ for such $P_i$).  Then after reindexing, we can assume that $\pi(P_i)\neq 0$, $i=1,\ldots, q'$, and that $P_{q'+1},\ldots, P_q$ have positive $r$th coordinate.  Since $\frac{q'}{r-1}=\frac{q}{r}$ as well, we can apply the inductive hypothesis to $\pi(P_1),\ldots, \pi(P_{q'})\in \mathbb{R}^{r-1}\setminus \{0\}$.  It follows that there exist disjoint subsets $I'_1,\ldots, I'_{\frac{q'}{r-1}}\subset\{1,\ldots, q'\}$ of cardinality $r-1$ such that
\begin{align*}
\sum_{i\in I_j'}\pi(P_i)
\end{align*}
has positive coordinates in $\mathbb{R}^{r-1}$ for $j=1,\ldots, \frac{q'}{r-1}=\frac{q}{r}$.  Let $I_j=I_j'\cup \{q'+j\}$, $j=1,\ldots, \frac{q}{r}$.  Then 
\begin{align*}
\sum_{i\in I_j}P_i
\end{align*}
has positive coordinates for $j=1,\ldots, \frac{q}{r}$ as desired.
\end{proof}

Lemma \ref{thmblem} has the following consequence in the context of Theorem \ref{mthm}.

\begin{proposition}
Let $X$ be a projective variety.  Let $E_1,\ldots, E_r$ be nef Cartier divisors on $X$ with $\sum_{j=1}^rE_j$ ample. Let $D_1,\ldots, D_q$ be non-zero effective (possibly reducible) Cartier divisors in general position on $X$, and suppose that $D_i\equiv \sum_{j=1}^ra_{i,j}E_j$, $i=1,\ldots, q$, where the coefficients $a_{i,j}$ are non-negative real numbers.  Let $P_i=(a_{i,1},\ldots, a_{i,r})\in \mathbb{R}^r$, $i=1,\ldots, q$.    Assume that for any proper subset $T$ of the set of standard basis vectors $\{e_1,\ldots, e_r\}\subset\mathbb{R}^r$, at most $(\#T)\left\lfloor\frac{q}{r}\right\rfloor$ of the vectors $P_1,\ldots, P_q$ are supported on $T$.  Then there exist $q'= \left\lfloor \frac{q}{r}\right\rfloor$ ample effective divisors $A_1,\ldots, A_{q'}$ in general position on $X$ with support contained in the support of $\sum_{i=1}^qD_i$.
\end{proposition}

\begin{proof}
Let $I_1,\ldots, I_{q'}\subset \{1,\ldots, q\}$ be as in Lemma \ref{thmblem} (with respect to $P_1,\ldots, P_q$) and let
\begin{align*}
A_m=\sum_{i\in I_m}D_i, \quad m=1,\ldots, q'.
\end{align*}
Since the divisors $D_1, \ldots, D_q$ are in general position on $X$ and the sets $I_m$ are pairwise disjoint, it is elementary that the divisors $A_1,\ldots, A_{q'}$ are in general position on $X$.  Moreover, since $\sum_{j=1}^rE_j$ is ample, $E_1,\ldots, E_r$ are nef divisors, and by construction, $A_m$ is numerically equivalent to a positive linear combination of $E_1,\ldots, E_r$, it follows that each divisor $A_m$ is ample.
\end{proof}

Theorem \ref{mthm}\eqref{mthmb} is now an immediate consequence of the preceding proposition combined appropriately with Theorem \ref{NW}, as the rank of the subgroup in $\NS X$ generated by $D_1,\ldots, D_q$ is no greater than the number of nef divisors $E_1,\ldots,E_r$ in the assumptions of Theorem \ref{mthm}. Moreover, Theorem \ref{mthm2} is now an immediate consequence of  Theorem \ref{Autissier} and Theorem \ref{Levin2}. Here, we use the fact that if $X\setminus E$ is arithmetically (quasi-)hyperbolic and $\Supp E\subset \Supp D$, then $X\setminus D$ is arithmetically (quasi-)hyperbolic.

\section{Examples}
\label{examples}

We first give two examples showing that in Vojta's Theorems \ref{Vojta1} and \ref{Vojta2} the conclusions cannot, in general, be strengthened to quasi-hyperbolicity statements.  In the first example the divisors $D_i$ are ample, but not in general position, and in the second example the divisors $D_i$ are in general position, but are not ample. 

\begin{example}
\label{P2lines}
Let $X=\mathbb{P}^2$ and let $D$ be a sum of at least $4$ lines passing through a fixed point $P\in \mathbb{P}^2(k)$. Theorems \ref{Vojta1} and \ref{Vojta2} imply that any set of $(D,S)$-integral points is not Zariski dense in $\mathbb{P}^2$ (in fact, by Siegel's theorem, this already holds when $D$ consists of the sum of just $3$ lines passing through $P\in \mathbb{P}^2(k)$, as in this case $\mathbb{P}^2\setminus D\cong \mathbb{A}^1\times (\mathbb{P}^1\setminus \{0,1,\infty\})$).  On the other hand, it is easy to see that any line $L$ through $P$ not contained in the support of $D$ contains an infinite set of $(D,S)$-integral points (for some $k$ and $S$).  Thus, $X\setminus D$ is not arithmetically quasi-hyperbolic.
\end{example}

\begin{example}
\label{P1squared}
Let $X=\mathbb{P}^1\times \mathbb{P}^1$ and let $D$ be a sum of at least $5$ fibers of the first natural projection.  Theorems \ref{Vojta1} and \ref{Vojta2} imply that any set of $(D,S)$-integral points is not Zariski dense in $\mathbb{P}^1\times \mathbb{P}^1$ (again, $3$ fibers are actually sufficient from an $S$-unit equation argument).  On the other hand, it is easy to see that any fiber of the first projection (not contained in the support of $D$) contains an infinite set of $(D,S)$-integral points (for some $k$ and $S$).  Then $X\setminus D$ is not arithmetically quasi-hyperbolic.
\end{example}

Next we give a sample application of Theorem \ref{mthm}\eqref{mthma} which does not seem to follow na\"ively from other previous results.

\begin{example}
\label{appex}
Let $X$ be a non-singular projective variety of dimension $n$, defined over a number field $k$, with nef effective divisors $E_1,E_2,E_3$ on $X$ such that $E_1+E_2+E_3$ is ample, but $E_i+E_j$ is not ample, or even big, for all $i,j\in \{1,2,3\}$ (for instance, one could take $A$ an ample effective divisor on a non-singular projective $Y$, let $X=Y^3$, and let $E_i=\pi_i^*A$, $i=1,2,3$, where $\pi_i$ is the $i$th natural projection map $\pi_i:X\to Y$).  Let $D_{i,j,k}$ be an effective divisor numerically equivalent to some positive (rational) linear combination $a_{i,j,k}E_i+b_{i,j,k}E_j$ for $1\leq i<j\leq 3$, $1\leq k\leq n+2$.  Suppose that the $3n+6$ effective divisors $D_{i,j,k}$ are in general position on $X$ and let $D=\sum_{i,j,k}D_{i,j,k}$.  Then by Theorem \ref{mthm}\eqref{mthma}, $X\setminus D$ is arithmetically quasi-hyperbolic.  

It does not seem straightforward to deduce this consequence, in general, from earlier results without using arguments similar to the present ones.  For instance, with Autissier's Theorem \ref{Autissier} in mind, there is not a way to generate more than $\frac{3}{2}n+3$ ample effective divisors in general position from the $D_{i,j,k}$ (assuming they are irreducible) nor (in view of \cite[Th.~3.2]{levin_duke}) a way to generate $n+2$ numerically equivalent ample effective divisors in general position from the $D_{i,j,k}$ (for general choices of $a_{i,j,k}$ and $b_{i,j,k}$).  We emphasize that the arithmetic quasi-hyperbolicity of $X\setminus D$ is the key aspect here (Zariski non-density of integral points follows easily from, say, Vojta's Theorem \ref{Vojta2}).

The above example can be naturally extended to the case of arbitrary
$r\geq 3$ and thus shows that our result is genuinely new. On the
other hand, with some additional considerations in the spirit of Lemma
\ref{mainl}, the cases $r=1,2$ of Theorem \ref{mthm}\eqref{mthma} may
be reduced to \cite[Th.~3.2]{levin_duke}).

\end{example}

The last two examples concern the sharpness of Conjecture \ref{conj2}.

\begin{example}
\label{conjaex}
Let $r$ and $n$ be positive integers with $r\leq n$, and let $Y_1,\ldots, Y_{r-1}$ be codimension $2$ linear spaces in $\mathbb{P}^n$ defined over a number field $k$ and in general position (i.e., all intersections among them have the expected dimension).  Let $\pi:X_{n,r}\to\mathbb{P}^n$ be the blowup along $Y_1\cup\cdots\cup Y_{r-1}$.  Let $H_{1,j},\ldots, H_{n-r+2,j}$, $j=1,\ldots, r$ be hyperplanes over $k$ passing through $Y_j$ for $j=1,\ldots, r-1$ (with no such condition when $j=r$) and let $H_{i,j}'$ be the strict transform of $H_{i,j}$ in $X_{n,r}$, $i=1,\ldots, n-r+2, j=1,\ldots, r$.  We additionally choose the hyperplanes $H_{i,j}$ so that the divisors $H_{i,j}'$ are in general position on $X_{n,r}$ and let $D_{n,r}=\sum_{i,j}H_{i,j}'$.  We prove by induction on $r$ that $X_{n,r}\setminus D_{n,r}$ is not arithmetically quasi-hyperbolic.  If $r=1$ then $X_{n,r}=\mathbb{P}^n$ and $D$ is a sum of $n+1$ hyperplanes in general position.  It is well-known that in this case for any appropriate finite set of places $S$ with $|S|>1$ there is a Zariski-dense set of $(D_{n,r},S)$-integral points on $X_{n,r}$.

If $r>1$, let $H$ be a hyperplane containing $Y_{r-1}$, distinct from any hyperplane $H_{i,j}$, and let $H'$ be its strict transform in $X_{n,r}$ (note that $H'\cap H'_{i,r-1}=\emptyset$, $i=1,\ldots, n-r+2$).  For a general choice of $H$ (subject to the condition that it contains $Y_{r-1}$), $H'\setminus D_{n,r}$ is isomorphic to a variety of the form $X_{n-1,r-1}\setminus D_{n-1,r-1}$ (for a suitable choice of parameters).  Then by induction, $H'\setminus D_{n,r}$ is not arithmetically quasi-hyperbolic.  As the union of such $H'$ is Zariski dense in $X_{n,r}$ we find that $X_{n,r}\setminus D_{n,r}$ is not arithmetically quasi-hyperbolic.  Finally, we note that $D_{n,r}$ is a sum of $r(n-r+2)$ divisors which are easily checked to satisfy the hypotheses of Theorem \ref{mthm} (with the choice $E_j=H'_{1,j}$, $j=1,\ldots, r$).  Thus, in Conjecture \ref{conj2}\eqref{conja}, if $r\leq n$ then for the conclusion to hold it is necessary at least that $q\geq r(n-r+2)+1$.
\end{example}

\begin{example}
\label{conjbex}
Let $T=\{P_1,\ldots, P_s, Q,R\}$ be a set of distinct collinear points in $\mathbb{P}^n(k)$ lying on a line $L$.  Let $H_1,\ldots, H_{2n(s+1)}$ be hyperplanes over $k$ in $\mathbb{P}^n$ such that each $H_i$ contains exactly one point in $T$, the intersection of any $n+1$ of the hyperplanes is contained in $T$, and
\begin{align*}
\bigcap_{i=(j-1)(2n)+1}^{j(2n)}H_i&=\{P_j\}, \quad j=1,\ldots, s,\\
\bigcap_{i=(2s)n+1}^{(2s+1)n}H_i&=\{Q\}, \\
\bigcap_{i=(2s+1)n+1}^{(2s+2)n}H_i&=\{R\}.
\end{align*}
Let $\pi:X\to \mathbb{P}^n$ be the blowup at the $s$ points $P_1,\ldots, P_s$ and let $D_i$ be the strict transform of $H_i$ in $X$, $i=1,\ldots, 2n(s+1)$.  Let $r=s+1$.  Then the divisors $D_1,\ldots, D_{2nr}$ are easily seen to satisfy the hypotheses of Theorem \ref{mthm}, where $E_i=D_{2in}$, $i=1,\ldots, r$.  Let $L'$ denote the strict transform of $L$.  Then $L'$ intersects $D=\sum_{i=1}^{2nr}D_i$ only in the points $\pi^{-1}(Q)$ and $\pi^{-1}(R)$, and so $X\setminus D$ admits a non-constant morphism from $\mathbb{G}_m$. It follows that $X\setminus D$ is not arithmetically hyperbolic and that the inequality in Conjecture \ref{conj2}\eqref{conjb} is sharp (if true).
\end{example}


\begin{thebibliography}{Sch76b}

\bibitem[Aut09]{Aut09}
P.~Autissier.
\newblock G\'eom\'etrie, points entiers et courbes enti\`eres.
\newblock {\em Ann. Sci. \'Ec. Norm. Sup\'er. (4)}, 42(2):221--239, 2009.

\bibitem[Aut11]{Aut11}
P.~Autissier.
\newblock Sur la non-densit\'e des points entiers.
\newblock {\em Duke Math. J.}, 158(1):13--27, 2011.

\bibitem[BG06]{BG}
E.~Bombieri and W.~Gubler.
\newblock {\em Heights in {D}iophantine geometry}.
\newblock New Mathematical Monographs, volume~4.
\newblock Cambridge University Press, Cambridge, 2006.

\bibitem[CLZ09]{CLZ}
P. Corvaja, A. Levin, and U. Zannier.
\newblock Integral points on threefolds and other varieties.
\newblock {\em Tohoku Math. J. (2)}, 61(4):589--601, 2009. 

\bibitem[CZ02]{CZ02}
P. Corvaja and U. Zannier.
\newblock A subspace theorem approach to integral points on curves.
\newblock {\em C. R. Math. Acad. Sci. Paris}, 334(4):267--271, 2002. 

\bibitem[CZ04a]{CZ}
P. Corvaja and U. Zannier.
\newblock On a general Thue's equation.
\newblock {\em Amer. J. Math.}, 126:1033--1055, 2004. 

\bibitem[CZ04b]{CZ04b}
P. Corvaja and U. Zannier.
\newblock On integral points on surfaces.
\newblock {\em Ann. of Math. (2)}, 160(2):705--726, 2004. 

\bibitem[EF02]{ef_imrn}
J.-H. Evertse and R.~Ferretti.
\newblock Diophantine inequalities on projective varieties.
\newblock {\em Int. Math. Res. Not.}, 2002(25):1295--1330, 2002.

\bibitem[EF08]{ef_festschrift}
J.-H. Evertse and R.~Ferretti.
\newblock A generalization of the {S}ubspace {T}heorem with polynomials of
  higher degree.
\newblock In {\em Diophantine approximation}, volume~16 of {\em Dev. Math.},
  pages 175--198. SpringerWienNewYork, Vienna, 2008.

\bibitem[Fer00]{ferretti_compositio}
R.~Ferretti.
\newblock Mumford's degree of contact and {D}iophantine approximations.
\newblock {\em Compositio Math.}, 121(3):247--262, 2000.

\bibitem[HT01]{HT01}
B.~Hassett and Y.~Tschinkel.
\newblock Density of integral points on algebraic varieties.
\newblock In {\em Rational points on algebraic varieties}, volume 199 of {\em
  Progr. Math.}, pages 169--197. Birkh\"auser, Basel, 2001.

\bibitem[HL17]{HL17}
G.~Heier and A. Levin.
\newblock A generalized {S}chmidt subspace theorem for closed subschemes. 
\newblock {\em Amer. J. Math.} (to appear),  
\newblock arXiv:1712.02456.

\bibitem[Laz04]{PAGI}
R.~Lazarsfeld.
\newblock {\em Positivity in {A}lgebraic {G}eometry. {I}}, volume~48 of {\em
  Ergebnisse der Mathematik und ihrer Grenzgebiete. 3. Folge. A Series of
  Modern Surveys in Mathematics.}
\newblock Springer-Verlag, Berlin, 2004.

\bibitem[Lev09]{levin_annals}
A.~Levin.
\newblock Generalizations of {S}iegel's and {P}icard's theorems.
\newblock {\em Ann. of Math. (2)}, 170(2):609--655, 2009.

\bibitem[Lev14]{levin_duke}
A.~Levin.
\newblock On the {S}chmidt subspace theorem for algebraic points.
\newblock {\em Duke Math. J.}, 163(15):2841--2885, 2014.

\bibitem[Mor82]{Mori}
S.~Mori.
\newblock Threefolds whose canonical bundles are not numerically effective.
\newblock {\em Ann. of Math. (2)}, 116(1):133--176, 1982.

\bibitem[MR15]{McK_R}
D.~McKinnon and M.~Roth.
\newblock Seshadri constants, {D}iophantine approximation, and {R}oth's theorem
  for arbitrary varieties.
\newblock {\em Invent. Math.}, 200(2):513--583, 2015.

\bibitem[NW14]{NW}
J.~Noguchi and J.~Winkelmann.
\newblock {\em Nevanlinna theory in several complex variables and {D}iophantine approximation}.
\newblock Grundlehren der Mathematischen Wissenschaften [Fundamental Principles of Mathematical Sciences], volume 350.
\newblock Springer, Tokyo, 2014.

\bibitem[RW17]{RW}
M.~Ru and J.~T.-Y. Wang.
\newblock A subspace theorem for subvarieties.
\newblock {\em Algebra {\&} Number Theory}, 11(10):2323--2337, 2017.

\bibitem[Sch76a]{Schlickewei}
H.~P. Schlickewei.
\newblock Die {$p$}-adische {V}erallgemeinerung des {S}atzes von
  {T}hue-{S}iegel-{R}oth-{S}chmidt.
\newblock {\em J. Reine Angew. Math.}, 288:86--105, 1976.

\bibitem[Sch76b]{Schlickewei_Acta_Arith}
H.~P. Schlickewei.
\newblock On products of special linear forms with algebraic coefficients.
\newblock {\em Acta Arith.}, 31(4):389--398, 1976.

\bibitem[Sch70]{Schmidt_Acta}
W.~Schmidt.
\newblock Simultaneous approximation to algebraic numbers by rationals.
\newblock {\em Acta Math.}, 125:189--201, 1970.

\bibitem[Sch72]{Schmidt_Annals}
W.~Schmidt.
\newblock Norm form equations.
\newblock {\em Ann. of Math. (2)}, 96:526--551, 1972.

\bibitem[Sil87]{silverman_87}
J.~Silverman.
\newblock Arithmetic distance functions and height functions in {D}iophantine
  geometry.
\newblock {\em Math. Ann.}, 279(2):193--216, 1987.

\bibitem[Voj87]{Vojta_LNM}
P.~Vojta.
\newblock {\em Diophantine approximations and value distribution theory},
  volume 1239 of {\em Lecture Notes in Mathematics}.
\newblock Springer-Verlag, Berlin, 1987.

\bibitem[Voj89]{vojta_ajm_87}
P.~Vojta.
\newblock A refinement of {S}chmidt's subspace theorem.
\newblock {\em Amer. J. Math.}, 111(3):489--518, 1989.

\bibitem[Voj96]{vojta_inv_math_1996}
P.~Vojta.
\newblock Integral points on subvarieties of semiabelian varieties. {I}.
\newblock {\em Invent. Math.}, 126(1):133--181, 1996.

\end{thebibliography}
\end{document}